\documentclass[twocolumn]{autart}    

\usepackage{cite}
\usepackage{amssymb,bm,mathtools,amsmath,amsfonts}
\usepackage{graphicx}
\usepackage{textcomp}
\usepackage{xcolor}

\newenvironment{proof}{\pf}{$\square$ \endpf}

\newcommand{\trace}{\mathtt{tr}}
\newcommand{\Expectation}{\mathbb{E}}
\renewcommand{\Re}{\mathbb{R}}

\allowdisplaybreaks[1]

\begin{document}

\begin{frontmatter}

\title{Recursively Feasible Stochastic Model Predictive Control for Time-Varying Linear Systems Subject to Unbounded Disturbances\thanksref{footnoteinfo}}

\author[Jacob]{Jacob W. Knaup\thanksref{correspondingAuthor}}\ead{jacobk@gatech.edu},
\author[Panos]{Panagiotis Tsiotras}\ead{tsiotras@gatech.edu}
\address[Jacob]{Georgia Institute of Technology, School of Interactive Computing and Institute for Robotics and Intelligent Machines, Atlanta, GA, USA}
\address[Panos]{Georgia Institute of Technology, School of Aerospace Engineering and Institute for Robotics and Intelligent Machines, Atlanta, GA, USA}

\thanks[footnoteinfo]{This paper was not presented at any conference.}
\thanks[correspondingAuthor]{Corresponding author.}



\begin{abstract}
    Model predictive control solves a constrained optimization problem online in order to compute an implicit closed-loop control policy. 
    Recursive feasibility---guaranteeing that the optimal control problem will have a solution at every time step---is an important property to guarantee the success of any model predictive control approach. 
    However, recursive feasibility is difficult to establish in a stochastic setting and, in particular, in the presence of disturbances having unbounded support (e.g., Gaussian noise).
    The problem is further exacerbated for time-varying systems, in which case recursive feasibility must be established also in a robust sense, over all possible future time-varying parameter values, as well as in a stochastic sense, over all potential disturbance realizations. 
    This work presents a method for ensuring the recursive feasibility of a convex, affine-feedback stochastic model predictive control problem formulation for systems with time-varying system matrices and unbounded disturbances using ideas from covariance steering stochastic model predictive control.
    It is additionally shown that the proposed approach ensures the closed-loop operation of the system will satisfy the desired chance constraints in practice, and that the stochastic model predictive control problem may be formulated as a convex program so that it may be efficiently solved in real-time.
\end{abstract}

\end{frontmatter}

\section{Introduction}

Model predictive control (MPC) is an attractive optimization-based method for controlling autonomous systems owing to its ability to handle input and state constraints while ensuring convergence under real-time computational limitations~\cite{borrelli2017predictive}. 
Stochastic MPC (SMPC), in particular, has gained interest as a method of providing probabilistic guarantees of constraint satisfaction and convergence to a neighborhood of the target in the presence of random disturbances.
However, several challenges of stochastic MPC still exist.
Whereas deterministic MPC optimizes over control actions at each time step, stochastic MPC must, instead, optimize over control \emph{policies} as a means of controlling the dispersion of trajectories over time. 
Since optimizing over arbitrary policies is not easily tractable, typical formulations consider an affine feedback policy \cite{mesbah2016stochastic}. 
In some cases, the feedback gain is calculated offline and is fixed; a more optimal approach is to design the feedback policy online in conjunction with a feed-forward control action. 
This, however, can lead to challenges with preserving the convexity of the problem formulation. 
Another issue of SMPC is ensuring convergence. 
Whereas for deterministic MPC it is possible to show asymptotic convergence for the case of linear systems, for SMPC only convergence to a neighborhood of the origin can be shown since the system is subject to a persistent disturbance \cite{mayne2014model}.
Finally, providing recursive feasibility guarantees for SMPC approaches is challenging when the noise is sampled from a distribution with unbounded support (e.g., Gaussian distribution). 
Existing approaches addressing this issue typically consider either a truncated Gaussian distribution or relax the problem formulation~\cite{mesbah2016stochastic, cannon2010stochastic, kouvaritakis2010explicit, farina2013probabilistic, hewing2018stochastic}.

In our recent work \cite{knaup2023safe}, we formulated a stochastic MPC approach for linear \emph{time-varying} systems based on the recent theoretical developments in the area of finite-horizon optimal covariance steering~\cite{goldshtein2017finite, chen2015optimal1, chen2015optimal2, chen2018optimal, bakolas2016optimal, bakolas2018finite, okamoto2018optimal, okamoto2019optimal}.
The covariance steering problem addresses the issue of how to optimally drive the first two moments of a linear system from some initial condition to a given final condition over a finite horizon. 
This problem is closely related to the questions of recursive feasibility and stability of SMPC, since bounds on the first two moments of the state distribution may be used to ensure constraint satisfaction and convergence. 
Using results from covariance steering, in \cite{knaup2023safe} we formulated an online affine-feedback solution to the SMPC problem that can be efficiently solved using convex programming, which we refer to as covariance steering stochastic MPC (CS-SMPC). 
Assuming recursive feasibility, the method provides asymptotic tracking convergence for the nominal system and convergence to a neighborhood of the reference trajectory for the stochastic system. 
However, the question of ensuring the recursive feasibility, and consequently closed-loop constraint satisfaction, of CS-SMPC was not addressed in~\cite{knaup2023safe} and is the focus of this paper. 

\subsection{Literature Review}

Most papers addressing the issue of recursive feasibility for stochastic systems assume that the disturbance is drawn from a distribution with bounded support (e.g., truncated Gaussian distribution). 
In \cite{cannon2010stochastic, kouvaritakis2010explicit, cannon2012stochastic, kouvaritakis2015developments}, the authors consider a linear time-invariant (LTI) system subject to bounded additive disturbances and design an SMPC controller based on a fixed stabilizing feedback gain which is computed offline. 
In \cite{oldewurtel2008tractable, korda2011strongly}, the authors consider an SMPC problem for an LTI system subject to an additive Gaussian disturbance. 
The authors in these papers optimize over the space of affine feedback policies; however, they also approximate the Gaussian with a truncated distribution, essentially formulating a robust MPC control problem. 

One of the few works that considered disturbances drawn from a distribution with unbounded support is \cite{farina2013probabilistic}, in which the authors optimize over the space of initial conditions and affine feedback policies in order to achieve recursive feasibility. 
The drawback of this strategy is that if the current state would lead to infeasibility, an open loop control must be used, 
discarding the most recent state information. 
The authors of \cite{farina2013probabilistic} extended their method to consider output feedback in \cite{farina2015approach}. 
However, the method of ensuring recursive feasibility remains similar, and suffers from the same drawbacks, as in~\cite{farina2013probabilistic}. 

A more recent work \cite{kohler2022recursively} notes that the relaxation of the initialization proposed in \cite{farina2013probabilistic} is still the state-of-the-art and has been utilized in numerous theoretical extensions and applications. 
The authors of \cite{kohler2022recursively} propose to improve the initialization strategy by allowing for a continuous interpolation between using the most recent state measurement and relying on open-loop prediction of the state. 
They prove recursive feasibility and convergence for their method; however, their method still relies on relaxing the initialization constraint and may discard the most recent state information. 
Additionally, the authors in \cite{kohler2022recursively} use a fixed feedback gain computed offline rather than optimizing over affine-feedback policies.
The authors of \cite{hewing2018stochastic, hewing2020recursively} also extended the work of \cite{farina2013probabilistic}. 
However, instead of optimizing over the state initialization, they propose to include the state measurement only in the cost function while the constraints are given purely in terms of the nominal system through a tube-based constraint-tightening approach. 
Unfortunately, this approach can be overly conservative due to the reliance on a pre-computed feedback policy and does not take advantage of the most recent information in applying the state and input constraints. 

The authors of \cite{paulson2020stochastic} also consider unbounded disturbances; however, they take a different approach. 
They address the issue of recursive feasibility and establish convergence for an unconstrained variation of the problem by softening the chance constraints using a penalty function (i.e., by imposing them through the cost function). 
These existing results all require relaxing the problem's constraints (either the current state initialization or the safety constraints) in order to prove recursive feasibility.

Very few works deal with SMPC for linear time-varying (LTV) systems, and those that do, only consider bounded disturbances. 
Considerable challenges hamper the direct application of existing recursive feasibility results to LTV systems. 
For example, the results obtained in \cite{farina2013probabilistic} require a pre-computed stabilizing feedback gain to design the terminal constraints. 
However, determining a single feedback gain that is stabilizing over all possible time-varying future system parameters is challenging. 
Likewise, the authors of \cite{paulson2020stochastic} assume that the system's state-transition matrix is Schur or Lyapunov stable which is often not the case in practice and can be difficult to verify ahead of time for an LTV system. 
Therefore, there is a need for results addressing LTV systems.
One of the few works addressing LTV systems is \cite{gonzalez2011online} which considers a robust MPC problem formulated with an online affine-feedback policy. 
However, as mentioned above, the authors limit their analysis to bounded disturbances. 
The authors of \cite{hanema2017stabilizing, lima2017stability} consider a linear parameter-varying system that approximates a nonlinear deterministic system. 
They show that recursive feasibility can be ensured using robust tube-based and invariant set-based approaches, respectively, provided the parameters lie in a compact set. 

As summarized in Table~\ref{tab:recursive_feasibility_literature}, the existing literature on the recursive feasibility of SMPC for systems subject to unbounded disturbances is extremely sparse, especially for the less-conservative case in which an affine-feedback policy is computed online. Furthermore, the overlap between results that consider both unbounded disturbances \emph{and} time-varying systems is non-existent. 
In this work, we fill this gap in the literature by establishing sufficient conditions for the recursive feasibility of SMPC for LTV systems subject to unbounded disturbances. 
Moreover, we provide a practical approach for computing the terminal constraint necessary in order to ensure that these conditions hold in practice.
We show that our method results in satisfaction of the chance constraints during closed-loop operation of the system as desired.
Finally, we show that these results hold for the particular formulation of the SMPC problem used by CS-SMPC in \cite{knaup2023safe}, which optimizes online over the space of affine feedback policies by solving a convex program.

\begin{table}[h]
\centering
\caption{Comparison of existing results on recursive feasibility of stochastic MPC.}
\label{tab:recursive_feasibility_literature}
\begin{tabular}{|l||c|c|}
\hline
Disturbances & \begin{tabular}[c]{@{}c@{}}Bounded (Robust)\\ \cite{oldewurtel2008tractable, cannon2010stochastic, kouvaritakis2010explicit, korda2011strongly, cannon2012stochastic, kouvaritakis2015developments} \\
\cite{gonzalez2011online, hanema2017stabilizing, lima2017stability}
\end{tabular} 
& \begin{tabular}[c]{@{}c@{}}Unbounded \\ (Stochastic)\\ \cite{farina2013probabilistic, farina2015approach, paulson2020stochastic, hewing2018stochastic, hewing2020recursively, kohler2022recursively} \end{tabular}   \\
\hline

Feedback     & \begin{tabular}[c]{@{}c@{}}Offline (Tube)\\ \cite{cannon2010stochastic,  kouvaritakis2010explicit, hewing2018stochastic, cannon2012stochastic} \\ 
\cite{kouvaritakis2015developments, hewing2020recursively, kohler2022recursively} \end{tabular}   
& \begin{tabular}[c]{@{}c@{}}Online \\ (Affine-feedback)\\ \cite{oldewurtel2008tractable, gonzalez2011online, korda2011strongly, farina2013probabilistic, farina2015approach, paulson2020stochastic}\end{tabular} \\
\hline

System       & \begin{tabular}[c]{@{}c@{}}LTI\\ \cite{cannon2010stochastic, cannon2012stochastic, farina2013probabilistic, farina2015approach, hewing2020recursively, hewing2018stochastic} \\ \cite{korda2011strongly, kouvaritakis2010explicit, kouvaritakis2015developments, paulson2020stochastic, oldewurtel2008tractable, kohler2022recursively} \end{tabular}             
& \begin{tabular}[c]{@{}c@{}}LTV\\ \cite{gonzalez2011online, hanema2017stabilizing, lima2017stability}\end{tabular}                      \\
\hline
%
%
\end{tabular}
\end{table}

\section{Problem Formulation}

We consider a linear time-varying (LTV) system given by
\begin{align} \label{eq:system_dynamics}
    x_{k+1} &= A_k x_k + B_k u_k + D_k w_k + r_k,
\end{align}
where $x_k \in \Re^{n_x}$ and $u_k \in \Re^{n_u}$ are the fully observable state and control action at time-step $k$, respectively, and $w_k \sim \mathcal{N}(\bm{0}_{n_w}, I_{n_w})$ is an i.i.d. disturbance (that is, $\Expectation[w_{k_1} w_{k_2}^\top] = \Expectation[w_{k_1}]\Expectation[w_{k_2}^\top] = 0$ for all $k_1 \neq k_2$). 
The system parameter matrices/vector are given by $A_k \in \mathcal{A} \subseteq \Re^{n_x \times n_x}$, $B_k \in \mathcal{B} \subseteq \Re^{n_x \times n_u}$, $D_k \in \mathcal{D} \subseteq \Re^{n_x \times n_w}$, and $r_k \in \mathcal{R} \subseteq \Re^{n_x}$, where $\mathcal{A}$, $\mathcal{B}$, $\mathcal{D}$, and $\mathcal{R}$ are such that 
\begin{align} \label{eq:convex_hull}
    S_k &= \begin{bmatrix}
        A_k & B_k &
        D_k & r_k
    \end{bmatrix} \in \mathcal{S} = \mathrm{co} \{S^{1}, \dots, S^{N_{p}} \} \nonumber\\
    &= \{ \sum_{\ell=1}^{N_{p}} \lambda_{\ell} S^{\ell} : \sum_{\ell=1}^{N_{p}} \lambda_{\ell} = 1,~ \lambda_{\ell} \geq 0 \},
\end{align}
where 
$S^{\ell} = \begin{bmatrix}
    A^{\ell} & B^{\ell} &
    D^{\ell} & r^{\ell}
\end{bmatrix}$, for all $\ell = 1, \dots, N_{p}$.

The system \eqref{eq:system_dynamics} is subject to $N_x$ and $N_u$ polytopic state and input chance constraints, respectively, given by
\begin{subequations} \label{eq:chance_constraints}
    \begin{align}
        \Pr(\alpha_{x, i}^\top x_k &\leq \beta_{x, i}) \geq 1 - p_{x, i}, \\
        \Pr(\alpha_{u, j}^\top u_k &\leq \beta_{u, j}) \geq 1 - p_{u, j}, 
    \end{align}
\end{subequations}
for $i = 1, 2, \dots, N_x$ and $j = 1, 2, \dots, N_u$, where the linear inequalities define polytopes given by
\begin{subequations} \label{eq:constraint_sets}
    \begin{align}
	\mathcal{X} &\triangleq \bigcap_{i = 1}^{N_x} \left\{x: \alpha_{x,i}^\top x \leq \beta_{x,i}\right\},
	\label{eq:X_definition} 
	\\
	\mathcal{U} &\triangleq 
	\bigcap_{j = 1}^{N_u} \left\{u: \alpha_{u,j}^\top u \leq \beta_{u,j}\right\}.
	\label{eq:U_definition}
    \end{align}
\end{subequations}
The states and control actions of system \eqref{eq:system_dynamics} are associated with a cost function given by 
\begin{align}
    \ell_k(x_k, u_k) = (x_k - x_{g_k})^\top Q_k (x_k - x_{g_k}) + u_k^\top R_k u_k,
\end{align}
where, for all $k = 0, 1, \dots$, $Q_k \succeq 0$ and $R_k \succ 0$ are known bounded matrices, and $x_{g_k} \in \mathcal{X}$ is a known target sequence to track.

To design an optimal control policy 
for system \eqref{eq:system_dynamics}, define the following Covariance Steering SMPC (CS-SMPC) problem given by
\begin{subequations} \label{prob:smpc}
    \begin{align}
        &\min_{\bm{\rho}_{k}^{N-1}(\cdot)} J_N(\mu_k, \Sigma_k; \bm{\rho}_{k}^{N-1}(\cdot)) = \sum_{t=k}^{k+N-1} \Expectation[ \ell_t(x_{t|k}, u_{t|k})], \\
        &\text{subject to} \nonumber\\
        &x_{t+1|k} = A_t x_{t|k} + B_t u_{t|k} + D_t w_t + r_t, \\
        &x_{k|k} \sim \mathcal{N}(\mu_k, \Sigma_k), \label{const:initial_state}\\
        & u_{t|k} = \rho_{t|k}(x_{t|k}), \\
        &\Pr(\alpha_{x, i}^\top x_{t|k} \leq \beta_{x, i}) \geq 1 - p_{x, i}, ~i=1, \dots, N_x, \label{const:state_chance}\\
        &\Pr(\alpha_{u, j}^\top u_{t|k} \leq \beta_{u, j}) \geq 1 - p_{u, j}, ~j=1, \dots, N_u, \label{const:control_chance}\\
        &\Expectation[x_{k+N|k}] \in \mathcal{X}_f^\mu, \label{const:terminal_mean} \\
        &\Expectation[(x_{k+N|k} - \Expectation[x_{k+N|k}])(x_{k+N|k} - \Expectation[x_{k+N|k}])^\top] \preceq \Sigma_f, \label{const:terminal_cov}
    \end{align}
\end{subequations}
where, for all $t = k, \dots, k+N-1$, $x_{t|k}$ and $u_{t|k}$ are the predicted state and control action, respectively, at time $t$ computed at the present time $k$ and where $\mathcal{X}_f^\mu \subseteq \mathcal{X}$ and $\Sigma_f \succ 0$ must be designed to ensure the recursive feasibility of \eqref{prob:smpc}. 
The optimization is performed over the sequence of planned state-feedback control \emph{policies} 
$\bm{\rho}_{k}^{N-1}(\cdot) = \{\rho_{k|k}(\cdot), \dots, \rho_{k+N-1|k}(\cdot) \}$. 
%
Since the future state sequence $\bm{x}_{k}^{N-1}
= \{x_{k|k}, \dots, x_{k+N-1|k} \}$
consists of random variables, the sequence of predicted future control actions, given by
$\bm{u}_{k}^{N-1} = \{u_{k|k}, \dots, u_{k+N-1|k} \}$, where $u_{t|k} = \rho_{t|k}(x_{t|k})$, is also stochastic.

Problem~\eqref{prob:smpc} is a finite-horizon optimal control problem representing a tractable approximation of the corresponding infinite-horizon optimal control problem, and is solved in a receding horizon fashion to generate an optimal closed-loop control policy for the system.
Let the optimal cost of Problem~\eqref{prob:smpc} at time $k$ be denoted by $J_N^{\ast}(\mu_k, \Sigma_k)$ and let the corresponding optimal solution be denoted by
$\bm{\rho}_{k}^{\ast N-1}(\cdot) = \{\rho_{k|k}^{\ast}(\cdot), \dots, \rho_{k+N-1|k}^{\ast}(\cdot) \}$, 
generating the predicted sequences of stochastic control actions 
$\bm{u}_{k}^{\ast N-1} = \{u_{k|k}^\ast, \dots, u_{k+N-1|k}^\ast \}$
and states
$\bm{x}_{k}^{\ast N} = \{x_{k|k}^{\ast}, \dots, x_{k+N|k}^{\ast} \}$, 
where 
$u_{t|k}^{\ast} = \rho_{t|k}^{\ast}(x_{t|k}^{\ast})$, 
$x_{t+1|k}^{\ast} = A_{t} x_{t|k}^{\ast} + B_{t} u_{t|k}^{\ast} + D_{t} w_{t} + r_{t}$, 
and 
$x_{k|k}^{\ast} \sim \mathcal{N}(\mu_k, \Sigma_k)$. 
The sequence $\bm{x}_{k}^{\ast N}$ has the corresponding optimal predicted moment sequences denoted by 
$\bm{\mu}_{k}^{\ast N} = \{\mu_{k|k}^{\ast}, \dots, \mu_{k+N|k}^{\ast} \}$ 
and 
$\bm{\Sigma}_{k}^{\ast N} = \{\Sigma_{k|k}^{\ast}, \dots, \Sigma_{k+N|k}^{\ast} \}$, 
where $x^{\ast}_{t|k} \sim \mathcal{N}(\mu_{t|k}^{\ast}, \Sigma_{t|k}^{\ast})$ and where $\mu_{k|k}^{\ast} = \mu_k$ and $\Sigma_{k|k}^{\ast} = \Sigma_k$.
The control command applied to system \eqref{eq:system_dynamics} at time $k$ is then given by 
\begin{align} \label{eq:mpc_control_policy}
    u_k = u_{k|k}^\ast.
\end{align}

In~\cite{knaup2023safe}, we showed that Problem~\eqref{prob:smpc} can be cast as a convex program for a suitable parameterization of the control policy (see Theorem 2 of \cite{knaup2023safe}). 
Additionally, we showed that if system \eqref{eq:system_dynamics} is controlled by \eqref{eq:mpc_control_policy}, then the nominal system ($w_k \equiv 0$) converges asymptotically to the reference trajectory $x_{g_k}$ (see Theorem 3 of \cite{knaup2023safe}) and the stochastic system \eqref{eq:system_dynamics} converges to a neighborhood of the reference trajectory on average (that is, the average stage cost is bounded as $k \rightarrow \infty$, see Theorem 4 of \cite{knaup2023safe}). 
However, both of these results rely on the condition that a solution to \eqref{prob:smpc} exists at every time $k = 0, 1, \dots$, using the initialization \eqref{const:initial_state}
where $\mu_k$ and $\Sigma_k$ are given by
\begin{equation} \label{eq:initialization}
    (\mu_k, \Sigma_k) =
        \begin{cases}
           (x_0, 0), & k = 0, \\[5pt]
            (\mu^\ast_{k|k-1}, \Sigma^\ast_{k|k-1}), & k > 0,
        \end{cases}
\end{equation}
where 
\begin{subequations} \label{eq:moment_dynamics}
    \begin{align}
        \mu_{k+1|k}^\ast &= A_{k} \mu_{k} + B_{k} \bar{u}_k + r_{k}, \\
        \Sigma_{k+1|k}^\ast &= A_k \Sigma_{k} A_k^\top + A_k \Sigma_{xu_k} B_k^\top + B_k \Sigma_{xu_k}^\top A_k^\top \nonumber\\
        &+ B_k \Sigma_{u_k} B_k^\top + D_k D_k^\top,
    \end{align}
\end{subequations}
and where $\bar{u}_k = \Expectation[u_{k}|\mu_k, \Sigma_k]$, $\Sigma_{xu_k} = \Expectation[(x_k - \mu_k]) (u_k - \bar{u}_k)^\top | \mu_k, \Sigma_k]$ and $\Sigma_{u_k} = \Expectation[(u_k - \bar{u}_k) (u_k - \bar{u}_k)^\top | \mu_k, \Sigma_k]$. 
Indeed, all existing stability results for SMPC rely on the assumption that the
corresponding optimal control problem  is 
feasible at each time step. 

The aim of this work is to investigate conditions implying the feasibility of Problem~\eqref{prob:smpc} under the initialization \eqref{eq:initialization}.
%
%
In Section~\ref{sec:sufficient_conditions}, we derive sufficient conditions for $\mathcal{X}_f^\mu$ and $\Sigma_f$ such that Problem~\eqref{prob:smpc} is recursively feasible, and in Section~\ref{sec:terminal_constraints} we show how to compute a $\mathcal{X}_f^\mu$ and $\Sigma_f$ to meet these conditions.
We then show that our method of handling the terminal constraints allows for a convex problem formulation with an affine feedback parameterization and ensures satisfaction of the chance constraints in practice.


\begin{rem}
    In \cite{knaup2023safe} we allowed Problem~\eqref{prob:smpc} to be initialized with $\mu_k = x_k$, $\Sigma_k = 0$ instead of \eqref{eq:initialization} if the following two conditions are satisfied: Problem~\eqref{prob:smpc} is feasible under such an initialization, and also its solution provides a lower optimal cost than that under \eqref{eq:initialization} (i.e., $J_{N}^{\ast}(x_k, 0) \leq J_{N}^{\ast}(\mu_{k|k-1}^{\ast}, \Sigma_{k|k-1}^{\ast})$). 
    Therefore, feasibility under \eqref{eq:initialization} implies feasibility under the initialization proposed in \cite{knaup2023safe} and thus it provides a sufficient condition for the stability results in Theorems~3~and~4 of~\cite{knaup2023safe}.
\end{rem}

\section{Recursive Feasibility Conditions} \label{sec:sufficient_conditions}

We are now ready to introduce the main result of the paper, which establishes sufficient conditions for the recursive feasibility of Problem~\eqref{prob:smpc}.

\begin{thm} \label{thm:sufficient_conditions_recursive_feasibility}
    Let $\mathcal{X}_{f}^{\mu}$ and $\Sigma_f$ be such that,
    given $\mu \in \mathcal{X}_{f}^{\mu}$ and $\Sigma \preceq \Sigma_f$, 
    there exists a $\rho(\cdot)$
    such that, for all $A \in \mathcal{A}$, $B \in \mathcal{B}$, $D \in \mathcal{D}$, $r \in \mathcal{R}$,
        \begin{subequations} \label{eq:terminal_invariance}
    \begin{align}
        &\Pr(\alpha_{x, i}^\top x \leq \beta_{x, i}) \geq 1 - p_{x, i}, ~i=1, \dots, N_x, \label{eq:term_imp_state_chance} \\
        &\Pr(\alpha_{u, j}^\top u \leq \beta_{u, j}) \geq 1 - p_{u, j}, ~j=1, \dots, N_u, \label{eq:term_imp_control_invariance} \\
        &A \mu + B \bar{u} + r \in \mathcal{X}_{f}^{\mu}, \label{eq:term_imp_mean_inv} \\
        &A \Sigma A^\top + A \Sigma_{xu} B^\top + B \Sigma_{xu}^\top A^\top \nonumber\\
        &\qquad\quad + B \Sigma_{u} B^\top + D D^\top \preceq \Sigma_{f} \label{eq:term_imp_sigma_inv},
    \end{align}
    where 
    $x \sim \mathcal{N}(\mu, \Sigma)$,
    $u = \rho(x)$,
    $\bar{u} = \Expectation[u | \mu, \Sigma]$, $\Sigma_{xu} = \Expectation[(x - \mu) (u - \bar{u})^\top | \mu, \Sigma]$ and $\Sigma_{u} = \Expectation[(u - \bar{u}) (u - \bar{u})^\top | \mu, \Sigma]$.
    \end{subequations}
    Then, if Problem~\eqref{prob:smpc} is feasible at time $k=0$ using the initialization~\eqref{eq:initialization}, it will remain feasible under initialization~\eqref{eq:initialization} for all future time steps $k > 0$.
\end{thm}

\begin{proof}
    It is sufficient to show that feasibility at an arbitrary time-step $k$ implies feasibility at the next time-step $k+1$. 
    To that end, let an optimal solution $\bm{\rho}_{k}^{\ast N-1}(\cdot)$ to Problem \eqref{prob:smpc} at time $k$ exist with corresponding optimal control, state, and moment sequences $\bm{u}_{k}^{\ast N-1}$, $\bm{x}_{k}^{\ast N}$, $\bm{\mu}_{k}^{\ast N}$, and $\bm{\Sigma}_{k}^{\ast N}$.
    We claim that a feasible solution at time $k+1$ is given by
    \begin{align} \label{eq:feasible_policy_sequence}
        &\bm{\rho}_{k+1}^{N-1}(\cdot) = \{\rho_{k+1|k+1}(\cdot), \dots, \rho_{k+N|k+1}(\cdot) \} \nonumber\\
        &=  \{\rho_{k+1|k}^{\ast}(\cdot), \dots, \rho_{k+N-1|k}^{\ast}(\cdot), \rho(\cdot) \},
    \end{align}
    generating the control sequence
    \begin{align} \label{eq:feasible_control_sequence}
        &\bm{u}_{k+1}^{N-1} = \{u_{k+1|k+1}, \dots, u_{k+N|k+1} \}, 
    \end{align} 
    and the predicted state sequence 
    \begin{align}
        \bm{x}_{k+1}^{N} = \{x_{k+1|k+1}, \dots, x_{k+N+1|k+1} \},
    \end{align}
    with corresponding moment sequences
    \begin{align}
        &\bm{\mu}_{k+1}^{N} = \{\mu_{k+1|k+1}, \dots, \mu_{k+N+1|k+1} \} \nonumber\\
        &\quad = \{\mu_{k+1|k}^{\ast}, \mu_{k+2|k}^{\ast}, \dots, \mu_{k+N|k}^{\ast}, \mu_{k+N+1|k+1}\}, \\
        &\bm{\Sigma}_{k+1}^{N} = \{\Sigma_{k+1|k+1}, \dots, \Sigma_{k+N+1|k+1} \} \nonumber\\
        &\quad = \{\Sigma_{k+1|k}^{\ast}, \Sigma_{k+2|k}^{\ast}, \dots, \Sigma_{k+N|k}^{\ast}, \Sigma_{k+N+1|k+1}\},
    \end{align}
    where $u_{t|k+1} = \rho_{t|k+1}(x_{t|k+1})$, $x_{t|k+1} \sim \mathcal{N}(\mu_{t|k+1}, \Sigma_{t|k+1})$, for $t = k+1, \dots, k+N+1$.
    
    To show that \eqref{eq:feasible_policy_sequence} indeed results in a feasible solution of Problem~\eqref{prob:smpc}, 
    first, notice that under initialization \eqref{eq:initialization}, $x_{k+1|k+1} \sim \mathcal{N}(\mu_{k+1|k}^{\ast}, \Sigma_{k+1|k}^{\ast})$.
    Then, the sequences $\{u_{k+1|k+1}, \dots, u_{k+N-1|k+1} \}$, $\{x_{k+1|k+1}, \dots, x_{k+N-1|k+1} \}$ necessarily satisfy \eqref{const:state_chance} and \eqref{const:control_chance} as, due to \eqref{eq:initialization}, they will be distributed the same as those derived from a feasible solution of \eqref{prob:smpc}.
    Moreover, $x_{k+N|k+1}$ will be distributed according to $\mathcal{N}(\mu^{\ast}_{k+N|k}, \Sigma^{\ast}_{k+N|k})$ and satisfies \eqref{const:state_chance} because it necessarily satisfies the terminal constraints \eqref{const:terminal_mean}-\eqref{const:terminal_cov} as a solution of Problem~\eqref{prob:smpc}, which, under \eqref{eq:term_imp_state_chance}, implies satisfaction of \eqref{const:state_chance}.
    Finally, due to \eqref{eq:terminal_invariance}, and since, as a solution of \eqref{prob:smpc}, $\mu_{k+N|k}^{\ast} \in \mathcal{X}_{f}^{\mu}$ and $\Sigma_{k+N|k}^{\ast} \preceq \Sigma_f$, it is known that
    there exists $\rho$ such that $u_{k+N|k+1} = \rho_{k+N|k+1}(x_{k+N|k+1}) = \rho(x_{k+N|k+1})$ satisfying \eqref{const:control_chance}, and such that $x_{k+N+1|k+1}$ satisfies \eqref{const:terminal_mean}-\eqref{const:terminal_cov}.
\end{proof}

Thus, we have established sufficient conditions for the recursive feasibility of Problem~\eqref{prob:smpc} under the initialization~\eqref{eq:initialization}. 
Next, we show how \eqref{eq:terminal_invariance} can be used to design the terminal set $\mathcal{X}_{f}^{\mu}$ and the terminal covariance $\Sigma_f$ so as to ensure that the 
Problem~\eqref{prob:smpc} is recursively feasible.
As a direct consequence of that result, we will show that the 
convex CS-SMPC formulation presented in \cite{knaup2023safe} is also recursively feasible.

\section{Terminal Constraint Design} \label{sec:terminal_constraints}


To compute the terminal set $\mathcal{X}_{f}^{\mu}$ and the terminal covariance $\Sigma_f$ satisfying \eqref{eq:terminal_invariance}, we introduce the candidate control parameterization
\begin{align} \label{eq:state_feedback}
    u_{t|k} &= \rho_{t|k}(x_{t|k}) =  v_{t|k} + L_{t|k} (x_{t|k} - \mu_{t|k}),
\end{align}
where $v_{t|k} \in \Re^{n_u}$ and $L_{t|k} \in \Re^{n_u \times n_x}$. Under \eqref{eq:state_feedback}, $\bar{u}_{t|k} = v_{t|k}$, $\Sigma_{xu_{t|k}} = \Sigma_{t|k} L_{t|k}^{\top}$, and $\Sigma_{u_{t|k}} = L_{t|k} \Sigma_{t|k} L_{t|k}^{\top}$. 

\begin{lem} \label{lem:terminal_covariance}
    A terminal covariance matrix $\Sigma_f$ satisfying \eqref{eq:term_imp_sigma_inv} is given by 
    the solution to the semidefinite program~(SDP)
    \begin{subequations} \label{prob:terminal_covariance_SDP}
    \begin{align}
        &\min_{\tilde{\Sigma}, Z}~~\trace(\tilde{\Sigma}) \\
        &\text{subject to} \nonumber\\
        &
        \begin{bmatrix}
            \tilde{\Sigma} - D^{\ell} D^{\ell \top} & A^{\ell} \tilde{\Sigma} + B^{\ell} Z \\
            \tilde{\Sigma} A^{\ell \top} +  Z^\top B^{\ell \top} & \tilde{\Sigma}
        \end{bmatrix} \succeq 0, \label{const:LMI}
    \end{align}
    \end{subequations}
    for all $\ell = 1, \dots, N_{p}$, where $\tilde{\Sigma} \in \Re^{n_x \times n_x}$ and $Z \in \Re^{n_u \times n_x}$.
\end{lem}

\begin{proof}
    Under the control policy parameterization~\eqref{eq:state_feedback}, $\rho(x) = v + L(x - \mu)$,  \eqref{eq:term_imp_sigma_inv} reduces to the existence of a $L \in \Re^{n_u \times n_x}$, such that $(A + B L) \Sigma (A + B L)^\top + D D^\top \preceq \Sigma_{f}$, for all $A \in \mathcal{A}$, $B \in \mathcal{B}$, $D \in \mathcal{D}$, and $\Sigma \preceq \Sigma_f$.
    Thus, a terminal covariance matrix $\Sigma_f$ satisfying \eqref{eq:term_imp_sigma_inv}, is given by a matrix $\tilde{\Sigma}$ satisfying
    \begin{align} \label{eq:terminal_cov_lyap}
        \tilde{\Sigma} \succeq (A + B \tilde{L}) \tilde{\Sigma} (A + B \tilde{L})^{\top} + D D^{\top},
    \end{align}
    for all $A \in \mathcal{A}$, $B \in \mathcal{B}$, $D \in \mathcal{D}$ and for a particular choice of $\tilde{L} \in \Re^{n_u \times n_x}$.
    Similar to the approach in \cite{gonzalez2010adaptive}, using the Schur complement, \eqref{eq:terminal_cov_lyap} may be formulated as a linear matrix inequality (LMI) given by
    \begin{align} \label{eq:LMI}
        M = 
        \begin{bmatrix}
            \tilde{\Sigma} & (A + B \tilde{L}) & D \\
            (A + B \tilde{L})^{\top} & \tilde{\Sigma}^{-1} & 0 \\
            D^{\top} & 0 & I
        \end{bmatrix} \succeq 0,
    \end{align}
    for all $A \in \mathcal{A}$, $B \in \mathcal{B}$, $D \in \mathcal{D}$.
    
    %
    Using the Schur complement and \eqref{eq:convex_hull}, it may be seen that satisfaction of \eqref{eq:LMI} is ensured by
    \begin{align} \label{eq:schur_LMI}
        \begin{bmatrix}
            \tilde{\Sigma} - D^{\ell} D^{\ell \top} & (A^{\ell} + B^{\ell} \tilde{L}) \\
            (A^{\ell} + B^{\ell} \tilde{L})^{\top} & \tilde{\Sigma}^{-1} \\
        \end{bmatrix} \succeq 0,
    \end{align}    
    for all $\ell = 1, \dots, N_p$.
    Finally, \eqref{eq:schur_LMI} may be converted to the equivalent LMI given by $\eqref{const:LMI}$, which is convex in $\tilde{\Sigma}$ and $Z$, by multiplying from both sides by $\mathrm{blkdiag}(I, \tilde{\Sigma})$ and using a change of variables given by $Z = \tilde{L} \tilde{\Sigma}$.
    Thus, $\Sigma_f$ satisfying \eqref{eq:term_imp_sigma_inv} may be found by solving the SDP \eqref{prob:terminal_covariance_SDP}.
\end{proof}

Let the optimal solution of Problem~\eqref{prob:terminal_covariance_SDP} be given by $\tilde{\Sigma}^{\ast},~Z^{\ast}$. 
The optimal control gain $\tilde{L}^{\ast}$ may be recovered from $\tilde{L}^{\ast} = Z^{\ast} \tilde{\Sigma}^{{\ast}^{-1}}$, provided $\tilde{\Sigma}^{\ast} \succ 0$. 
%
We are now ready to introduce the second main result, which establishes how to compute $\mathcal{X}^{\mu}_{f}$ and $\Sigma_f$ so that the conditions of Theorem~\ref{thm:sufficient_conditions_recursive_feasibility} are satisfied.

\begin{thm} \label{thm:terminal_constraints}
    Let $\Sigma_f = \tilde{\Sigma}^{\ast}$ be given by \eqref{prob:terminal_covariance_SDP}, and let $\mathcal{X}_{f}^{\mu}$ be the largest set such that $\mathcal{X}_{f}^{\mu} \subseteq \mathcal{X}_{\mathrm{safe}}$ and such that, for all $\ell = 1, \dots, N_{p}$,
    \begin{align} \label{eq:terminal_set_comp}
        \mu \in \mathcal{X}_{f}^{\mu} \implies \exists~ v \in \mathcal{U}_{\mathrm{safe}} : A^{\ell} \mu + B^{\ell} v + r^{\ell} \in \mathcal{X}_{f}^{\mu},
    \end{align}
    where, 
    \begin{subequations} \label{eq:safe_sets}
    \begin{align}
        \mathcal{X}_{\text{safe}} &= \bigcap_{i = 1}^{N_x} \left\{\mu: \alpha_{x,i}^\top \mu \leq \tilde{\beta}_{x,i}\right\}, \label{eq:state_safe_set} \\
        \mathcal{U}_{\text{safe}} &= 
        \bigcap_{j = 1}^{N_u} \left\{v: \alpha_{u,j}^\top v \leq \tilde{\beta}_{u,j}\right\}, \label{eq:control_safe_set}
    \end{align}
    and where 
    \begin{align}
        &\hspace{-1.5mm} \tilde{\beta}_{x,i} = \beta_{x,i} - \sqrt{\alpha_{x, i}^{\top} \Sigma_{f} \alpha_{x, i}} \, \Phi^{-1}(1 - p_{x, i}), \label{eq:state_safe_tightening} \\
        &\hspace{-1.5mm} \tilde{\beta}_{u,j} = \beta_{u,j} - \sqrt{\alpha_{u, j}^{\top} \tilde{L}^{\ast} \Sigma_{f} \tilde{L}^{\ast \top} \alpha_{u, j}} \, \Phi^{-1}(1 - p_{u, j}). \label{eq:control_safe_tightening}
    \end{align}
    \end{subequations}
%
    Then, the conditions 
    \eqref{eq:terminal_invariance} of Theorem~\ref{thm:sufficient_conditions_recursive_feasibility} 
    for the recursive feasibility of Problem~\eqref{prob:smpc} are satisfied.
\end{thm}

\begin{proof}
    Let $\rho(x) = v + \tilde{L}^\ast (x - \mu)$ where $x \sim \mathcal{N}(\mu, \Sigma)$.
    Then, as shown in Lemma~3 of \cite{knaup2023safe}, the chance constraints \eqref{eq:term_imp_state_chance} and \eqref{eq:term_imp_control_invariance} may be reformulated as the deterministic constraints
    \begin{subequations} \label{eq:det_constraints}
        \begin{align}
            &\hspace{-1.5mm} \alpha_{x, i}^{\top} \mu + \sqrt{\alpha_{x, i}^{\top} \Sigma \alpha_{x, i}} \, \Phi^{-1}(1 - p_{x, i}) \leq \beta_{x, i} , \label{const:det_state_chance} \\
            &\hspace{-1.5mm} \alpha_{u, j}^{\top} v + \sqrt{\alpha_{u, j}^{\top} \tilde{L}^{\ast} \Sigma \tilde{L}^{\ast \top} \alpha_{u, j}} \, \Phi^{-1}(1 - p_{u, j}) \leq \beta_{u, j}, \label{const:det_control_chance}
        \end{align}
    \end{subequations}
    for all $i = 1, \dots, N_x$ and $j = 1, \dots, N_u$.
    It may be seen from \eqref{eq:safe_sets} that the deterministic constraints \eqref{eq:det_constraints} are satisfied for all $\Sigma \preceq \Sigma_f$, $\mu \in \mathcal{X}_{\text{safe}}$, and $v \in \mathcal{U}_{\text{safe}}$.
    Moreover, since $\mathcal{X}_{f}^{\mu} \subseteq \mathcal{X}_{\text{safe}}$, the chance constraints \eqref{eq:term_imp_state_chance} and \eqref{eq:term_imp_control_invariance} are  satisfied
    for all $\Sigma \preceq \Sigma_{f}$, $\mu \in \mathcal{X}_{f}^{\mu}$ and $v \in \mathcal{U}_{\text{safe}}$.
    Next, due to \eqref{eq:convex_hull} and \eqref{eq:state_feedback}, it may be seen that \eqref{eq:term_imp_mean_inv} is ensured for all $A \in \mathcal{A}$, $B \in \mathcal{B}$, and $r \in \mathcal{R}$ by the existence of a $v \in \Re^{n_u}$ such that $A^{\ell} \mu + B^{\ell} v + r^{\ell} \in \mathcal{X}_{f}^{\mu}$, for all $\ell = 1, \dots, N_p$ and for all $\mu \in \mathcal{X}^{\mu}_{f}$. Thus \eqref{eq:terminal_set_comp} is sufficient for satisfaction of \eqref{eq:term_imp_mean_inv}.
    Finally, by Lemma~\ref{lem:terminal_covariance}, setting $\Sigma_f$ to the optimal solution of \eqref{prob:terminal_covariance_SDP} ensures satisfaction of \eqref{eq:term_imp_sigma_inv}.
%
\end{proof}

\begin{rem}
    The set $\mathcal{X}_{f}^{\mu}$, given by \eqref{eq:terminal_set_comp}, is often referred to as a maximal robust controlled invariant set and may be efficiently computed using the iterative procedure presented in, e.g., \cite{borrelli2017predictive, gonzalez2011online}.
    Although, any robust controlled invariant set is sufficient, using the maximal such set minimizes conservatism as, ideally, the terminal set should be as large as possible while ensuring recursive feasibility.
\end{rem}

\subsection{Closed-loop Chance Constraint Satisfaction}

Next, we show that ensuring recursive feasibility using the proposed approach not only guarantees Problem~\eqref{prob:smpc} will have a solution, but also ensures that the control policy~\eqref{eq:mpc_control_policy}
will ensure the chance constraints~\eqref{eq:chance_constraints} will be satisfied in practice by the system under closed-loop operation, as desired.

\begin{thm} \label{thm:closed-loop_constraint_satisfaction}
    The control policy~\eqref{eq:mpc_control_policy} generated by solving Problem~\eqref{prob:smpc} with the terminal constraints $\Sigma_f$ and $\mathcal{X}_{f}^{\mu}$ given by \eqref{prob:terminal_covariance_SDP} and \eqref{eq:terminal_set_comp}, respectively, and using the initialization~\eqref{eq:initialization},
    ensures system~\eqref{eq:system_dynamics} satisfies
    \begin{subequations} \label{eq:chance_const_init_cond}
        \begin{align}
            \Pr(\alpha_{x, i}^\top x_k &\leq \beta_{x, i} | x_0) \geq 1 - p_{x, i}, \\
            \Pr(\alpha_{u, j}^\top u_k &\leq \beta_{u, j} | x_0) \geq 1 - p_{u, j}, 
        \end{align}
    \end{subequations}
    for $i = 1, 2, \dots, N_x$ and $j = 1, 2, \dots, N_u$ and $k = 0, 1, 2, \dots$, provided Problem~\eqref{prob:smpc} is feasible at time $k=0$.
\end{thm}
\begin{proof}
    Examining constraints~\eqref{const:initial_state},~\eqref{const:state_chance}, and~\eqref{const:control_chance}, it is clear that recursive feasibility of Problem~\eqref{prob:smpc} under Theorems~\ref{thm:sufficient_conditions_recursive_feasibility}~and~\ref{thm:terminal_constraints} and control policy~\eqref{eq:mpc_control_policy} imply
    \begin{subequations} \label{eq:closed-loop_chance_constraints}
    \begin{align}
        \Pr(\alpha_{x, i}^\top x_k &\leq \beta_{x, i} | \mu_k, \Sigma_k) \geq 1 - p_{x, i}, \\
        \Pr(\alpha_{u, j}^\top u_k &\leq \beta_{u, j} | \mu_k, \Sigma_k) \geq 1 - p_{u, j},
    \end{align}
    \end{subequations}
    for all $k = 0, 1, \dots$,
    where the problem is initialized with $x_k \sim \mathcal{N}(\mu_k, \Sigma_k)$.
    Moreover, under initialization~\eqref{eq:initialization} which depends on the moment dynamics~\eqref{eq:moment_dynamics} and the applied control policy~\eqref{eq:mpc_control_policy}, we have
    \begin{subequations}\label{eq:closed-loop_moment_dynamics}
    \begin{align} 
        \mu_{k+1} &= A_{k} \mu_{k} + B_{k} \bar{u}_{k}^{\ast} + r_{k}, \\
        \Sigma_{k+1} &= A_k \Sigma_{k} A_k^\top + A_k \Sigma_{xu_{k}}^{\ast} B_k^\top + B_k \Sigma_{xu_{k}}^{\ast \top} A_k^\top \nonumber\\
        &+ B_k \Sigma_{u_{k}}^{\ast} B_k^\top + D_k D_k^\top,
    \end{align}
    \end{subequations}
    for all $k = 0, 1, \dots$,
    where
    $\bar{u}_{k}^{\ast} = \Expectation[u_{k} | \mu_{k}, \Sigma_{k}]$, 
    $\Sigma_{xu_{k}}^{\ast} = \Expectation[(x_k - \mu_{k}) (u_k - \bar{u}_{k}^{\ast})^\top | \mu_{k}, \Sigma_{k}]$,
    and $\Sigma_{u_{k}}^{\ast} = \Expectation[(u_k - \bar{u}_{k}^{\ast}) (u_k - \bar{u}_{k}^{\ast})^\top | \mu_{k}, \Sigma_{k}]$,
    and where $\mu_0 = x_0$ and $\Sigma_0 = 0$.
    Note that \eqref{eq:closed-loop_moment_dynamics} are precisely the moment dynamics for the system~\eqref{eq:system_dynamics} operating in closed-loop under the control policy~\eqref{eq:mpc_control_policy}.
    Therefore, since the problem is initialized with $x_k \sim \mathcal{N}(\mu_k, \Sigma_k)$, given $x_0$, $\bm{\rho}_{k}^{\ast N-1}(\cdot)$ and the moment dynamics evolve deterministically.
    Thus,
    \begin{subequations}
        \begin{align}
            \mu_{k} &= \Expectation[x_{k} | \mu_0, \Sigma_0] = \Expectation[x_{k} | x_0], \\
            \Sigma_k &= \Expectation[(x_k - \mu_k) (x_k - \mu_k)^{\top} | \mu_0, \Sigma_0] \nonumber\\
            &= \Expectation[(x_k - \mu_k) (x_k - \mu_k)^{\top} | x_0],
        \end{align}
    \end{subequations}
    for all $k = 0, 1, \dots$.
    Therefore, \eqref{eq:closed-loop_chance_constraints} implies \eqref{eq:chance_const_init_cond}.
\end{proof}

\begin{cor}
    Under the conditions of Theorem~\ref{thm:closed-loop_constraint_satisfaction}, 
    system~\eqref{eq:system_dynamics} will additionally satisfy
    \begin{subequations} \label{eq:chance_const_bool}
        \begin{align}
            \Pr(x_k \in \mathcal{X} | x_0) \geq 1 - p_x, \\
            \Pr(u_k \in \mathcal{U} | x_0) \geq 1 - p_u,
        \end{align}
    \end{subequations}
    where $p_x \geq \sum_{i=1}^{N_x} p_{x, i}$ and $p_u \geq \sum_{j=1}^{N_u} p_{u, i}$.
\end{cor}
\begin{proof}
    The result follows immediately from Theorem~\ref{thm:closed-loop_constraint_satisfaction} by applying Boole's inequality~\cite{prekopa1988boole} and the definition of the polytopes $\mathcal{X}$ and $\mathcal{U}$ given in \eqref{eq:constraint_sets}.
\end{proof}

In some cases, (e.g. \cite{knaup2023safe, farina2013probabilistic}) it may be desirable to reduce conservatism by replacing the static initialization strategy~\eqref{eq:initialization} with a dynamic strategy
\begin{equation} \label{eq:dynamic_reinitialization}
    (\mu_k, \Sigma_k) =
    \begin{cases}
       (x_k, 0), & x_k \in \mathcal{X}_{\text{feas}}, \\[5pt]
        (\mu^\ast_{k|k-1}, \Sigma^\ast_{k|k-1}), & \text{otherwise},
    \end{cases}
\end{equation}
where $\mathcal{X}_{\text{feas}} \subseteq \mathcal{X}$ is the set of states for which Problem~\eqref{prob:smpc} is feasible with $\mu_k = x_k$ and $\Sigma_k = 0$.
In this case, \eqref{eq:chance_const_init_cond} is not inherently satisfied.
Rather, this initialization strategy ensures a different notion of closed-loop chance constraint satisfaction.
\begin{cor}
    The control policy~\eqref{eq:mpc_control_policy} generated by solving Problem~\eqref{prob:smpc} with the terminal constraints $\Sigma_f$ and $\mathcal{X}_{f}^{\mu}$ given by \eqref{prob:terminal_covariance_SDP} and \eqref{eq:terminal_set_comp}, respectively, and using initialization~\eqref{eq:dynamic_reinitialization},
    ensures recursive feasibility and that the closed-loop operation of system~\eqref{eq:system_dynamics} will satisfy 
    \begin{subequations} \label{eq:chance_const_recond}
        \begin{align}
            \Pr(\alpha_{x, i}^\top x_k &\leq \beta_{x, i} | x_{\tau}) \geq 1 - p_{x, i}, \\
            \Pr(\alpha_{u, j}^\top u_k &\leq \beta_{u, j} | x_{\tau}) \geq 1 - p_{u, j}, 
        \end{align}
    \end{subequations}
    for $i = 1, 2, \dots, N_x$ and $j = 1, 2, \dots, N_u$ and $k \geq \tau \geq 0$, provided $x_0 \in \mathcal{X}_{\text{feas}}$, and where $\tau$ is the most-recent time-step at which $x_\tau \in \mathcal{X}_{\text{feas}}$.
\end{cor}
\begin{proof}
    Recursive feasibility is ensured since the first case of initialization~\eqref{eq:dynamic_reinitialization} is ensured for $k=0$ and, due to Theorems~\ref{thm:sufficient_conditions_recursive_feasibility}~and~\ref{thm:terminal_constraints}, at least the second case will be ensured for $k > 0$.
    Closed-loop satisfaction of the chance-constraints~\eqref{eq:chance_const_recond} follows from a generalization of Theorem~\ref{thm:closed-loop_constraint_satisfaction} by replacing the fixed case $x_0$ with the dynamic case $x_\tau$ as the last state on which the problem was reconditioned.
\end{proof}

\subsection{Convex Formulation}

Theorem~\ref{thm:terminal_constraints} shows how to compute the terminal constraints under the control parameterization \eqref{eq:state_feedback}, so as to ensure recursive feasibility.
However, formulating Problem~\eqref{prob:smpc} with the affine state feedback parameterization \eqref{eq:state_feedback} results in a nonconvex program \cite{okamoto2018optimal, farina2013probabilistic}. Next, we show that by a change of variables, we may arrive at a convex reformulation of Problem~\eqref{prob:smpc}.

\begin{thm} \label{thm:equivalant_control}
    Any state-feedback policy of the form \eqref{eq:state_feedback} may be represented by an 
    affine feedback policy, given by
    \begin{subequations} \label{eq:aux_var_parameterization}
        \begin{align}
            u_{t|k} &= v_{t|k} + \sum_{i=k}^{t} K_{t, i|k} y_{t|k}, \\
            y_{t+1|k} &= A_t y_{t|k} + D_{t} w_{t}, \quad   y_{k|k} = x_{k} - \mu_{k},
        \end{align}
        where $y_{t|k} \in \Re^{n_x}$ is an auxiliary variable that tracks the error evolution of the uncontrolled system.
    \end{subequations}
    Moreover, using the parameterization \eqref{eq:aux_var_parameterization}, Problem~\eqref{prob:smpc} may be cast as a convex program.
\end{thm}

\begin{proof}
    The predicted dynamics of system \eqref{eq:system_dynamics} may be alternatively written as 
    \begin{align}\label{eq:aug_sys}
        X_{k|k}^{N} = \bar{A}_{k}^{N} x_{k|k} + \bar{B}_{k}^{N} U_{k|k}^{N-1} + \bar{D}_{k}^{N} W_{k}^{N-1} + \bar{r}_{k}^{N},
    \end{align}
    where 
    \begin{align*}
        X_{k|k}^{N} &= \begin{bmatrix}
        x_{k|k} \\ x_{k+1|k} \\ \vdots \\ x_{k+N|k}
        \end{bmatrix}, \quad
        U_{k|k}^{N-1} = \begin{bmatrix}
        u_{k|k} \\ u_{k+1|k} \\ \vdots \\ u_{k+N-1|k}
        \end{bmatrix},\\ 
        W_{k}^{N-1} &= \begin{bmatrix}
        w_{k} \\ w_{k+1} \\ \vdots \\ w_{k+N-1}
        \end{bmatrix},
    \end{align*}
    where
    $\bar{A}_{k}^{N} \in \Re^{(N+1)n_x \times n_x}$, $\bar{B}_{k}^{N} \in \Re^{(N+1)n_x \times N n_u}$, $\bar{D}_{k}^{N} \in \Re^{(N+1)n_x \times N n_w}$, $\bar{r}_{k}^{N} \in \Re^{(N+1)n_x}$, and $\bar{B}_{k}^{N}$ and $\bar{D}_{k}^{N}$ are lower block triangular. 
    The exact expressions for the system matrices $\bar{A}_{k}^{N}$, $\bar{B}_{k}^{N}$, $\bar{D}_{k}^{N}$ may be found in \cite{goldshtein2017finite}, for example.
    
    We introduce the parameterization \eqref{eq:state_feedback} and split \eqref{eq:aug_sys} into the nominal and error dynamics as follows 
    \begin{subequations}
        \begin{align}
            X_{k|k}^{N} &= \bar{X}_{k|k}^{N} + \tilde{X}_{k|k}^{N}, \\
            \bar{X}_{k|k}^{N} &= \bar{A}_{k}^{N} \mu_k + \bar{B}_{k}^{N} V_{k|k}^{N-1} + \bar{r}_{k}^{N}, \\
            \tilde{X}_{k|k}^{N} &= \bar{A}_{k}^{N} \tilde{x}_k + \bar{B}_{k}^{N} L_{k|k}^{N-1} \tilde{X}_{k|k}^{N} + \bar{D}_{k}^{N} W_{k}^{N-1} \nonumber\\
            &= (I - \bar{B}_{k}^{N} L_{k|k}^{N-1})^{-1} (\bar{A}_{k}^{N} \tilde{x}_k + \bar{D}_{k}^{N} W_{k}^{N-1}),
        \end{align}
        where $\tilde{x}_k = x_k - \mu_k$, and
        \begin{align}
            V_{k|k}^{N-1} &= \begin{bmatrix}
                v_{k|k} \\
                \vdots \\
                v_{k+N-1|k}
            \end{bmatrix}, \\
            L_{k|k}^{N-1} &= \begin{bmatrix}
                L_{k|k} & 0 & \dots & 0 & 0 \\
                0 & L_{k+1|k} & \dots & 0 & 0 \\
                \vdots & & \ddots & \vdots & \vdots \\
                0 & 0 & \dots & L_{k+N-1|k} & 0
            \end{bmatrix}.
        \end{align}
    \end{subequations}
    Therefore, 
    \begin{align}
        &U_{k|k}^{N-1} = V_{k|k}^{N-1} \nonumber\\
        &+ L_{k|k}^{N-1} (I - \bar{B}_{k}^{N} L_{k|k}^{N-1})^{-1} (\bar{A}_{k}^{N} \tilde{x}_k + \bar{D}_{k}^{N} W_{k}^{N-1}).
    \end{align}
    
    Introducing the change of variables 
    \begin{subequations} \label{eq:change_of_variables}
        \begin{align}
            K_{k|k}^{N-1} &= L_{k|k}^{N-1} (I - \bar{B}_{k}^{N} L_{k|k}^{N-1})^{-1}, \\
            Y_{k|k}^{N} &= (\bar{A}_{k}^{N} \tilde{x}_k + \bar{D}_{k}^{N} W_{k}^{N-1}),
        \end{align}
        yields the control parameterization
        \begin{align}
            U_{k|k}^{N-1} = K_{k|k}^{N-1} Y_{k|k}^{N} + V_{k|k}^{N-1},
        \end{align}
        where  $K_{k|k}^{N-1} =$ 
        \begin{align*}
           \hspace{-3mm}\begin{bmatrix}
                K_{k, k|k} & 0 & \dots & 0 & 0 \\
                K_{k+1, k|k} & K_{k+1, k+1|k} & \dots & 0 & 0 \\
                \vdots & & \ddots & \vdots & \vdots \\
                K_{k+N-1, k|k} &  K_{k+N-1, k+1|k} & \dots & K_{k+N-1, k+N-1|k} & 0
            \end{bmatrix},
        \end{align*}
    \end{subequations}
    where the lower triangular structure of $K_{k|k}^{N-1}$ results from the block lower triangular structure of 
   $\bar{B}_{k}^{N}$.
    
    The parameterization given by \eqref{eq:change_of_variables} is the augmented form of the affine feedback policy \eqref{eq:aux_var_parameterization}.
    In Theorem~2 of \cite{knaup2023safe}, we have shown that  Problem~\eqref{prob:smpc} may be cast as a convex program by introducing the affine feedback parameterization given by \eqref{eq:aux_var_parameterization}.
\end{proof}

\begin{rem}
    In \cite{knaup2023safe}, the feedback is applied only with respect to the current error (that is, $u_{t|k} = v_{t|k} + K_{t|k} y_{t|k}$). However, this is done merely to limit the number of decision variables and improve computation time. The proof of convexity in \cite{knaup2023safe} may clearly be extended to the more general case of \eqref{eq:aux_var_parameterization}.
\end{rem}

In summary, 
using Theorem~\ref{thm:sufficient_conditions_recursive_feasibility} we have established sufficient conditions to guarantee recursive feasibility of Problem~\eqref{prob:smpc}. Theorem~\ref{thm:terminal_constraints} shows how the terminal constraints may be designed so as to satisfy these conditions in practice.
Theorem~\ref{thm:closed-loop_constraint_satisfaction} establishes that the chance constraints will be satisfied under closed-loop operation as desired.
Finally, Theorem~\ref{thm:equivalant_control} shows that the results of the previous theorems extend to the convex problem formulation utilized in \cite{knaup2023safe}.

\section{Numerical Example}

We evaluate the proposed approach in a realistic vehicle lateral control example. 
We consider a nonlinear kinematic bicycle model, which, as in, e.g., \cite{wang2021path}, is linearized around a reference trajectory to obtain an LTV system. 
The nonlinear kinematic bicycle model, shown in Fig.~\ref{fig:bike}, is given by the following 
\begin{subequations}
    \begin{align}
        \dot{e}_{\psi} &= \dot{\psi} - \dot{\psi}_{\mathrm{ref}} + \dot{\nu}_y / \nu_x, \\
        \dot{e}_y &= \nu_y \cos{e_{\psi}} + \nu_x \sin{e_{\psi}}, \\
        \dot{s} &= (\nu_x \cos{e_{\psi}} - \nu_y \sin{e_{\psi}}) / (1 - e_{y} \rho),
    \end{align}
\end{subequations}
where
$\nu_y = \nu_x \delta \frac{\ell_r}{\ell_f + \ell_r}$,
$\dot{\psi} = \tan{\delta} \frac{\nu_x}{\ell_f + \ell_r}$,
and where
$e_{\psi}$ is the heading error with respect to the reference path,
$\psi$ is the vehicle's absolute heading,
$\psi_{\mathrm{ref}}$ is the heading of the reference path,
$\nu_y$ is the vehicle's lateral velocity,
$\nu_x$ is the vehicle's longitudinal velocity, 
$e_y$ is the vehicle's lateral error with respect to the reference path,
$s$ is the vehicle's longitudinal position with respect to the reference path,
$\rho$ is the curvature of the reference path,
$\delta$ is the front steering angle,
and $\ell_f$ and $\ell_r$ are the locations of the vehicle's center of mass.
\begin{figure}[ht]
    \centering
    \includegraphics[width=\linewidth, trim={433 333 500 167}, clip]{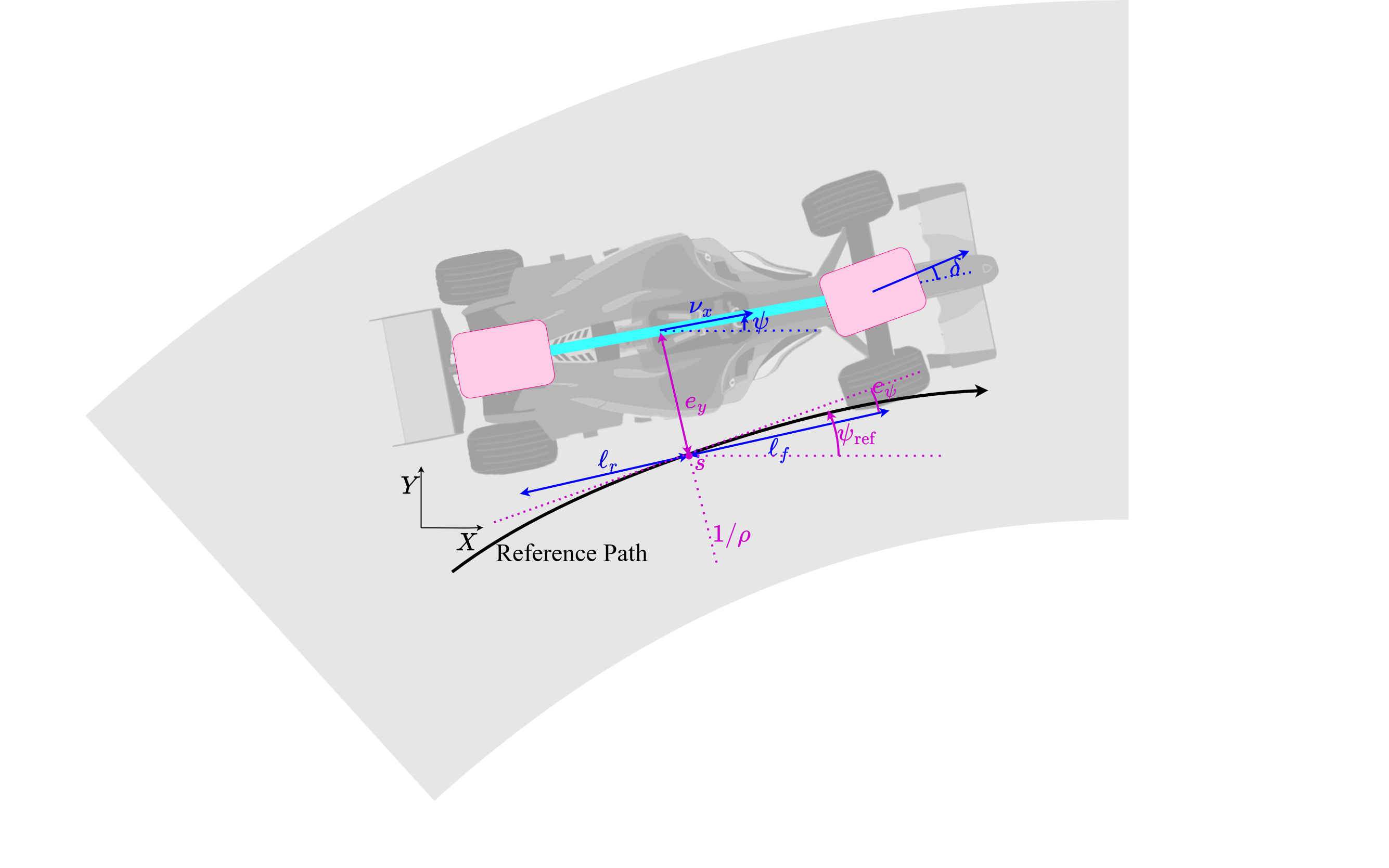}
    \caption{Kinematic Bicycle Model in Curvilinear Coordinates}
    \label{fig:bike}
\end{figure}
Given a velocity profile and a reference path, and assuming $\delta$ and $e_{\psi}$ remain small, a linear approximation for the lateral motion is given by
\begin{subequations}
    \begin{align}
        \dot{e}_{\psi} &= \frac{\nu_x}{\ell_f + \ell_r} \delta - \nu_x \rho + \dot{\delta} \frac{\ell_r}{\ell_f + \ell_r}, \\
        \dot{e}_{y} &= \frac{\ell_r}{\ell_f + \ell_r} \nu_x \delta + \nu_x e_{\psi}.
    \end{align}
\end{subequations}
We define the state $x = [\delta, e_{\psi}, e_y]^\top$ and control $u = \dot{\delta}$ and use Euler integration with a time-step of $\Delta t = 0.1$~s to obtain the affine system
\begin{align}
    &\hspace{-2mm} x_{k+1} = A(\nu_{x_k}) x_{k} + B(\nu_{x_k}) u_k + D w_k + r(\nu_{x_k}, \rho_k),
\end{align}
where, for the given reference trajectory $\bm{z_\mathrm{ref}} = \{[\nu_{x_0}, \rho_0]^\top, \dots, [\nu_{x_T}, \rho_T]^\top\}$,
\begin{subequations}
    \begin{align}
        A(\nu_{x_k}) &= A_k 
        = \begin{bmatrix}
            1 & 0 & 0 \\
            \frac{\nu_{x_k}} {\ell_f + \ell_r} \Delta t  & 1 & 0 \\
            \frac{\ell_r} {\ell_f + \ell_r} \nu_{x_k} \Delta t  & \nu_{x_k} \Delta t  & 1
        \end{bmatrix}, 
        \\
        B(\nu_{x_k}) &= B_k
        = \begin{bmatrix}
            \Delta t \\
            \frac{\ell_r} {\ell_f + \ell_r} \Delta t \\
            0
        \end{bmatrix},
        \\
        D &= \begin{bmatrix}
            \theta_{\delta} \Delta t & 0 & 0 \\
            0 & \theta_{\psi} \Delta t & 0 \\
            0 & 0 & \theta_{y} \Delta t
        \end{bmatrix},
        \\ 
        r(\nu_{x_k}, \rho_k) &= r_k
        = \begin{bmatrix}
            0 \\
            - \rho_k \nu_{x_k} \Delta t \\
            0 \\
        \end{bmatrix},
    \end{align}
\end{subequations}
where we have added the additive noise term to account for errors owing to the linear approximation.

We define 
\begin{subequations}
    \begin{align}
        \mathcal{X} &= \{x \in \Re^{3} : \begin{bmatrix}
            -\pi / 4 \\
            -\pi / 4 \\
            -2 \\
        \end{bmatrix} \leq x \leq \begin{bmatrix}
            \pi / 4 \\
            \pi / 4 \\
            2 \\
        \end{bmatrix}
        \},
    \end{align}
    and 
    \begin{align}
        \mathcal{U} &= \{u \in \Re : -1 \leq u \leq 1 \},
    \end{align}
\end{subequations}
and we let $\bm{z_\mathrm{ref}}$ satisfy, for all $z \in \bm{z_\mathrm{ref}}$,
\begin{align}
    \begin{bmatrix}
        \nu_\mathrm{min} \\
        \rho_\mathrm{min}
    \end{bmatrix}
    \leq z \leq 
    \begin{bmatrix}
        \nu_\mathrm{max} \\
        \rho_\mathrm{max}
    \end{bmatrix},
\end{align}
where $\nu_\mathrm{min} = 1.0$, $\nu_\mathrm{max} = 20.0$, $\rho_\mathrm{min} = -0.025$, $\rho_\mathrm{max} = 0.025$.
Thus, $\mathcal{S} = \mathrm{co}\{S(\nu_\mathrm{min}, \rho_\mathrm{min}),\allowbreak S(\nu_\mathrm{max}, \rho_\mathrm{min}),\allowbreak S(\nu_\mathrm{min},\allowbreak \rho_\mathrm{max}),\allowbreak S(\nu_\mathrm{max}, \rho_\mathrm{max}) \}$,
where $S(v, \rho) = [A(v),\allowbreak B(v),\allowbreak D, r(v, \rho)]$,
and $[A_k, B_k, D, r_k] \in \mathcal{S}$ for all $k = 0, \dots, T$.
We set the maximum probability of constraint violations to $p_{x, i} = 0.025$ for $i = 1, \dots, 6$ and $p_{u, j} = 0.05$ for $j = 1, 2$. 
The length of the control horizon is $N = 4$, the cost matrices are given by $Q = \mathrm{diag}(1, 1, 1)$, $R = 100$, and $x_{g_k} = \bm{0}_{3 \times 1}$ for all $k = 0, 1, \dots, T$.
The vehicle dimensions are $\ell_f=\ell_r = 2.4$, and we set the noise parameters
as $\theta_\delta = \theta_\psi = \theta_y = 0.1$.

The proposed approach is compared with two na\"{\i}ve implementations of CS-SMPC.
The first uses nominal terminal sets given by $\mathcal{X}_{f, \mathrm{nom}}^{\mu}$ and $\Sigma_{f, \mathrm{nom}}$ computed using an LTI system given by taking the mean values of $\nu_{x_k}$ and $\rho_k$ over $k = 0, \dots, T$,
and the second uses no terminal constraints at all. 
The robust and nominal terminal covariances ($\Sigma_f$ and $\Sigma_{f, \mathrm{nom}}$, respectively) are shown in Fig.~\ref{fig:terminal_covariance_vehicle}.
The three approaches are simulated over 20 Monte Carlo trials, with
the results, including the computed robust terminal set as well as the closed-loop CS-SMPC trajectories, shown in Figs.~\ref{fig:mpc_vehicle_error}-\ref{fig:mpc_vehicle_road}.

\begin{figure}[h]
    \centering
    \includegraphics[width=\columnwidth]{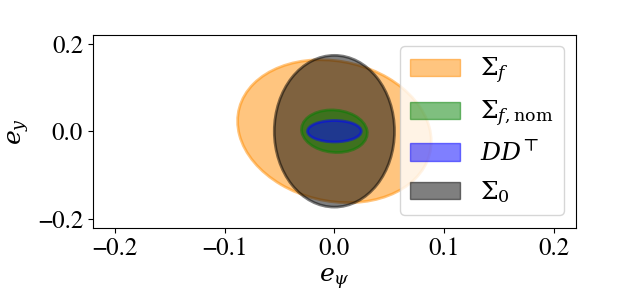}
    \caption{Computed robust terminal covariance $\Sigma_f$, compared with a nominal terminal covariance for a single set of system parameters $\Sigma_{f, \mathrm{nom}}$, the noise covariance $D D^\top$, and the initial state covariance $\Sigma_0$.}
    \label{fig:terminal_covariance_vehicle}
\end{figure}

\begin{figure}[h]
    \centering
    \includegraphics[width=\columnwidth]{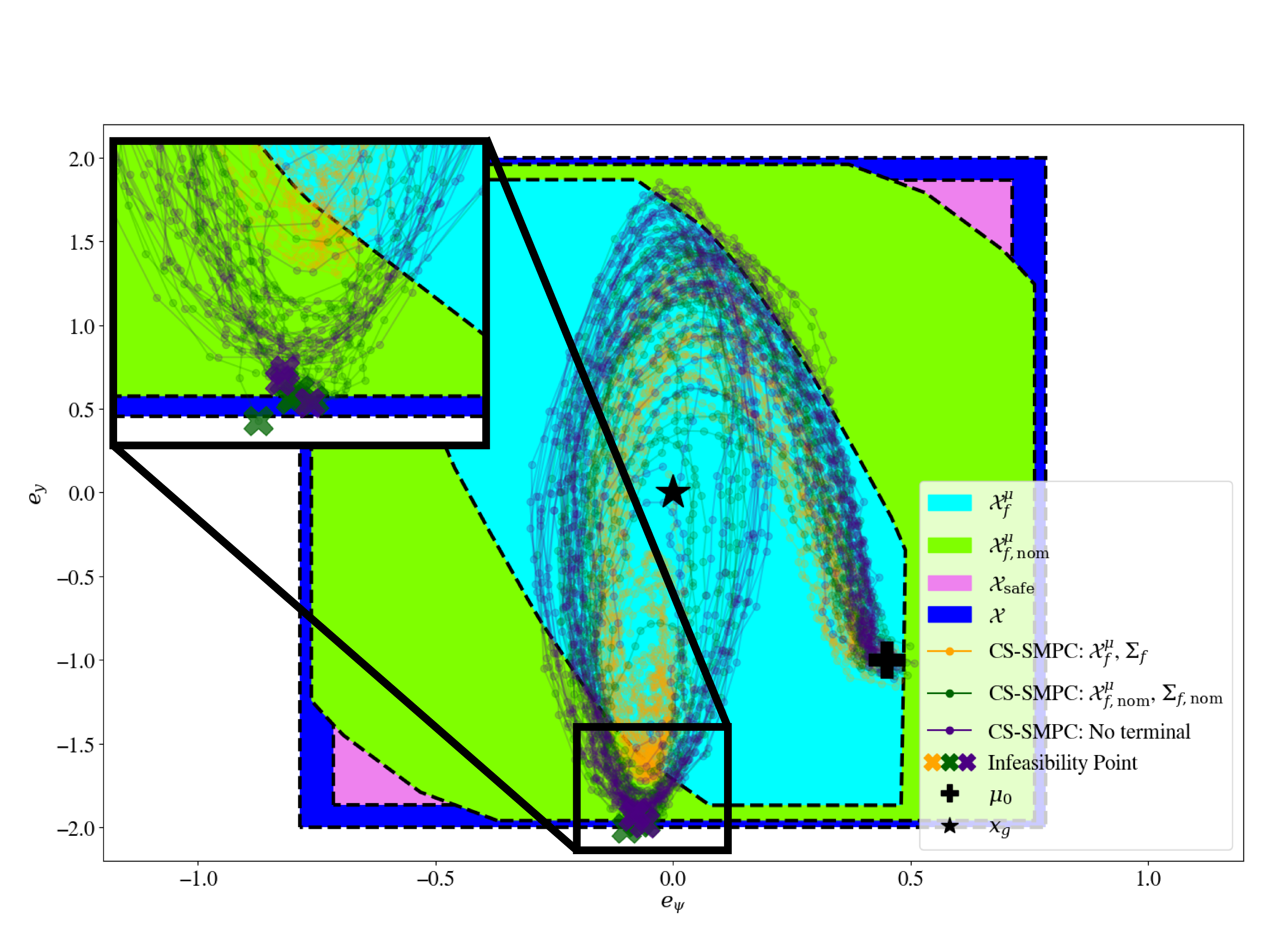}
    \caption{Computed robust terminal set $\mathcal{X}_{f}^{\mu}$, compared with a nominal terminal set for a single set of system parameters $\mathcal{X}^{\mu}_{f, \mathrm{nom}}$, and the tightened state constraints $\mathcal{X}_{\mathrm{safe}}$. 
    }
    \label{fig:mpc_vehicle_error}
\end{figure}

\begin{figure}[h]
    \centering
    \includegraphics[width=\columnwidth]{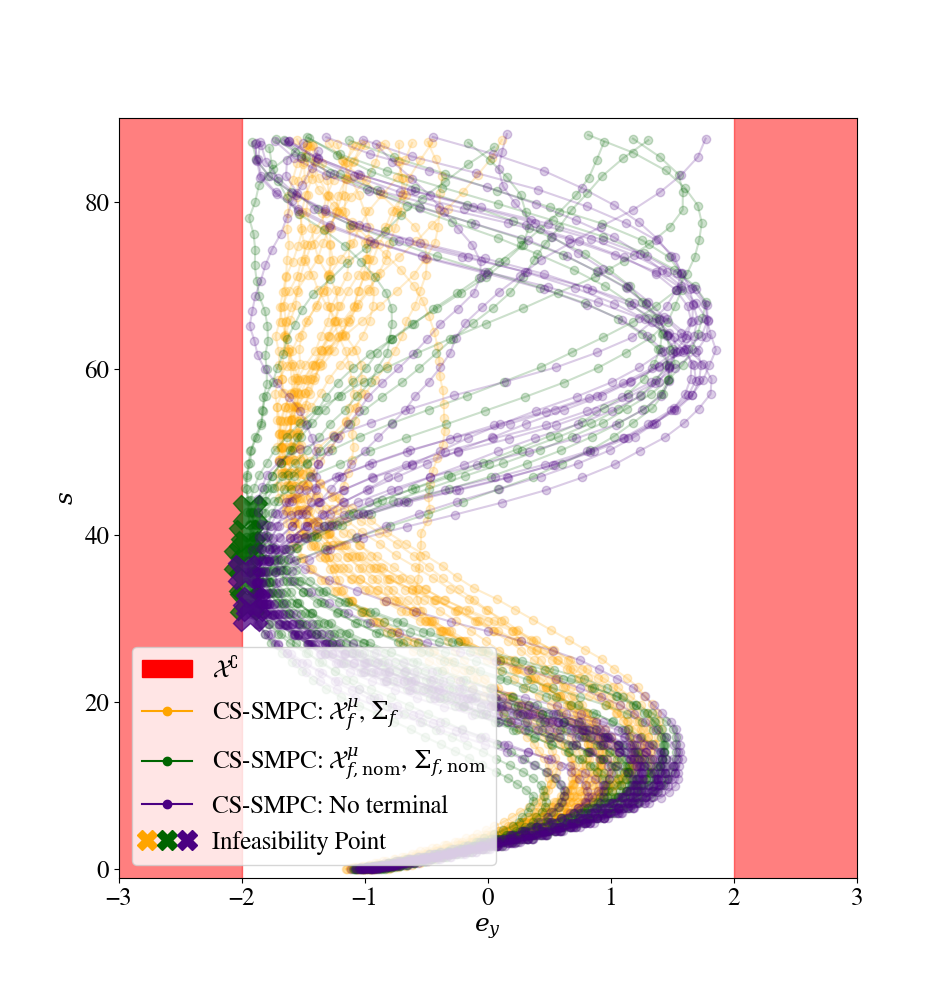}
    \caption{Monte Carlo closed-loop CS-SMPC trajectories in curvilinear coordinates with robust $\mathcal{X}_{f}^{\mu}$ and $\Sigma_f$, compared with CS-SMPC using nominal and no terminal constraints, respectively.}
    \label{fig:mpc_vehicle_road}
\end{figure}

It may be seen from Fig.~\ref{fig:mpc_vehicle_error} that the robust terminal set is a subset
of the nominal terminal set, $\mathcal{X}_{f}^{\mu} \subseteq \mathcal{X}_{f, \mathrm{nom}}^{\mu}$, as the robust terminal set is the intersection of all nominal control invariant sets for any system parameters. 
Likewise, the terminal sets $\mathcal{X}_{f}^{\mu}$ and $\mathcal{X}_{f, \mathrm{nom}}^{\mu}$ are subsets of $\mathcal{X}_{\text{safe}}$ as they are the sets for which any value of the mean state may be contained within $\mathcal{X}_{\text{safe}}$ at all future time steps. 

Figure~\ref{fig:mpc_vehicle_road} shows the closed loop trajectories in curvilinear coordinates for the vehicle using a control policy generated with CS-SMPC subject to the robust terminal constraints ($\mathcal{X}_{f}^{\mu}$, $\Sigma_{f}$), nominal terminal constraints ($\mathcal{X}_{f, \mathrm{nom}}^{\mu}$, $\Sigma_{f, \mathrm{nom}}$), or no terminal constraints, respectively. 
The colored crosses indicate points at which no feasible solution to the optimal control problem exists and the corresponding trajectory ends. 
This occurs when the vehicle reaches a state near the boundaries from which it cannot recover and no control policy exists within the search space that will prevent the vehicle from moving outside the boundaries and violating the state constraints in the near future. 
It may be seen, from the absence of orange crosses in Fig.~\ref{fig:mpc_vehicle_road}, that the closed-loop trajectories using the control policy \eqref{eq:mpc_control_policy} computed using CS-SMPC with the robust terminal constraints ($\mathcal{X}_{f}^{\mu}$, $\Sigma_{f}$) obey the state constraints and do not lead to infeasibility of Problem~\eqref{prob:smpc}. 
When CS-SMPC uses the incorrect terminal constraints ($\mathcal{X}_{f, \mathrm{nom}}^{\mu}$, $\Sigma_{f, \mathrm{nom}}$) or no terminal constraints at all, it may be seen from the green and purple crosses, that Problem~\eqref{prob:smpc} does not remain recursively feasible for all trajectories.
For the case of the nominal terminal constraints, this is due to the fact that $\mathcal{X}_{f, \mathrm{nom}}^{\mu}$ and $\Sigma_{f, \mathrm{nom}}$ are not invariant over all potential system matrices, so the future dynamics may render Problem~\eqref{prob:smpc} infeasible. 
For the case of no terminal constraints, clearly there is no invariance property and merely constraining the system to remain within $\mathcal{X}$ is insufficient to prevent the vehicle from reaching an irrecoverable state.
These results are summarized in Table~\ref{tab:infeasibilities}.

\begin{table}[h]
    \centering
    \caption{Infeasibility points encountered out of 20 trials for proposed method vs na\"{\i}ve baselines.}
    \begin{tabular}{|c||c|c|c|} \hline
        Terminal Constraints & $\mathcal{X}_f$ & $\mathcal{X}_{f, \mathrm{nom}}$ & None \\
        \hline
        Infeasibility Encountered & 0 & 6 & 3 \\
        \hline
    \end{tabular}
    \label{tab:infeasibilities}
\end{table}

\section{Conclusion}

In this work, we have established the first recursive feasibility result in the literature for SMPC for linear time-varying systems subject to unbounded disturbances. 
We first presented a theorem giving sufficient conditions for the recursive feasibility of the LTV SMPC optimization problem. 
Then, we introduced a result showing how the terminal constraints may be calculated to satisfy these sufficient conditions. 
Finally, we showed that these results extend to the Covariance Steering SMPC (CS-SMPC) problem used in our previous work, which solves for the feedback gains online using convex programming, and for which we have already established convergence guarantees subject to the condition of recursive feasibility. 
These results are verified via a numerical example using a nonlinear vehicle model.

\bibliographystyle{plain}
\bibliography{recursive_feasibility.bib}

\end{document}